\documentclass[11pt,a4paper]{article}
\usepackage[T1]{fontenc}
\usepackage[utf8]{inputenc}
\usepackage[margin=27mm]{geometry}
\usepackage{amsmath,amssymb,amsthm,mathtools,booktabs,microtype}
\usepackage[colorlinks=true,linkcolor=blue,citecolor=blue,urlcolor=blue]{hyperref}
\hypersetup{pdftitle={The minimum number of vertices of a triangulation of RP5},pdfauthor={Xinrui Chang}}
\allowdisplaybreaks
\newtheorem{theorem}{Theorem}[section]
\newtheorem{proposition}[theorem]{Proposition}
\newtheorem{lemma}[theorem]{Lemma}
\newtheorem{corollary}[theorem]{Corollary}
\theoremstyle{definition}
\newtheorem{definition}[theorem]{Definition}

\newcommand{\F}{\mathbb F_2}
\newcommand{\RP}{\mathbb{RP}}
\newcommand{\Tr}{\operatorname{Tr}}
\newcommand{\st}{\operatorname{st}}
\newcommand{\lk}{\operatorname{lk}}
\newcommand{\supp}{\operatorname{supp}}
\newcommand{\sd}{\operatorname{sd}}
\newcommand{\ip}[2]{\langle #1,#2\rangle}
\title{The minimum number of vertices of a triangulation of $\RP^5$}
\author{Xinrui Chang\\\small Fudan University\\\small\texttt{25800180002@m.fudan.edu.cn}}
\date{}
\begin{document}
\maketitle

\begin{abstract}
We give a computer-assisted proof that every finite simplicial triangulation of real projective five-space has at least twenty-four vertices. Together with the construction of Guyer, Steinerberger, and Yang, this determines the minimum. A cup-product argument produces an induced subcomplex on sixteen or seventeen vertices. Transfer to the spherical double cover then imposes a covering condition on its lifted faces. In the sixteen-vertex case, a direct capacity bound excludes twenty-two vertices. For twenty-three vertices, a cochain extension argument reduces every relevant completion to one of 115 critical signed complexes; integer weights and exact covering certificates exclude all of them. The argument requires no assumption that the triangulation is combinatorial.
\end{abstract}

\section{Introduction}
For a compact triangulable space $M$, let
\[
 \nu(M)=\min\{|V(K)|:K\text{ is a finite simplicial complex and }|K|\cong M\}.
\]
Arnoux and Marin~\cite[\S\S20--21, pp.~203, 207--208]{AM} proved that
\begin{equation}\label{eq:AM}
 \nu(\RP^d)\ge\binom{d+2}{2}+1\qquad(d\ge3).
\end{equation}
In particular, their result excludes every triangulation of $\RP^5$ with at most twenty-one vertices. Guyer, Steinerberger, and Yang~\cite{GSY} constructed a triangulation with twenty-four vertices as the antipodal quotient of the boundary of a centrally symmetric simplicial six-polytope. Thus the remaining vertex numbers are twenty-two and twenty-three.

\begin{theorem}\label{thm:main}
Every finite simplicial complex whose realization is homeomorphic to $\RP^5$ has at least twenty-four vertices. Consequently,
\[
 \nu(\RP^5)=24.
\]
\end{theorem}

The lower bound is proved by a finite classification and exact certificates. The twenty-two-vertex case uses only the unique signed complex on sixteen vertices with nonzero fourth cup power and a capacity bound of 48. The twenty-three-vertex case requires stronger covering obstructions and a classification of seventeen-vertex completions containing a six-vertex face. We give the topological reductions in full and specify the finite assertions and their verification below. All homology and cohomology groups have coefficients in $\F$.

\section{Signed complexes and cup powers}
For $W\subseteq V(K)$, write $K[W]$ for the induced subcomplex. Fix an ordering of the vertices. If $a\in Z^1(K;\F)$, its simplicial cup power is represented by
\begin{equation}\label{eq:cup}
 a^r[v_0,\ldots,v_r]=\prod_{i=0}^{r-1}a(v_iv_{i+1}).
\end{equation}
By the simplicial--singular comparison and the universal coefficient theorem over $\F$, $[a]^r\ne0$ if and only if some $z\in Z_r(K;\F)$ satisfies $\ip{a^r}{z}=1$. We call such a cycle positive.

\begin{definition}
A signature on $V$ assigns $s_{uv}\in\F$ to each unordered pair of distinct vertices. A subset $F\subseteq V$ is balanced if $s_{uv}+s_{vw}+s_{uw}=0$ for every triple in $F$. The balanced subsets form a simplicial complex $B(s)$. A switching replaces $s_{uv}$ by $s_{uv}+h_u+h_v$ for a function $h:V\to\F$. Signatures are equivalent if they differ by a switching and a permutation of vertices.
\end{definition}

A nonempty balanced set has exactly two potentials $t:F\to\F$ satisfying $s_{uv}=t_u+t_v$, related by complementation. These potentials describe its two lifts in the double cover associated with $s$. For a cocycle $a$ on $K$, assign arbitrary values to missing edges to obtain a signature $s$. Then
\begin{equation}\label{eq:completion}
 K\subseteq B(s),\qquad [a]^r\ne0\ \Longrightarrow\ [s]^r\ne0.
\end{equation}
The implication follows by restriction of cochains. Such a signature is called a compatible completion. Balanced faces in a completion need not be faces of $K$.

\begin{lemma}[Deletion]\label{lem:deletion}
Let $V(K)=A\sqcup D$, and let $\alpha\in H^1(K;\F)$. If $\alpha^r\ne0$, $r\ge2$, and $\alpha|_{K[A]}=0$, then $\alpha^{r-1}|_{K[D]}\ne0$.
\end{lemma}
\begin{proof}
Let $U_A$ be the open subset of $|K|$ on which the sum of the barycentric coordinates in $A$ is positive; define $U_D$ similarly. Coordinate normalization gives deformation retractions onto $|K[A]|$ and $|K[D]|$, respectively. If the conclusion failed, the two factors of $\alpha^r=\alpha\alpha^{r-1}$ would lift to $H^1(|K|,U_A)$ and $H^{r-1}(|K|,U_D)$. Their relative product lies in $H^r(|K|,U_A\cup U_D)=0$, a contradiction.
\end{proof}

\begin{lemma}[A facet intersection bound]\label{lem:hitting}
Suppose $a\in Z^1(K;\F)$ and $z\in Z_d(K;\F)$ satisfy $\ip{a^d}{z}=1$, where $d\ge1$. If a $d$-simplex $A$ meets every simplex in $\supp z$, then
\begin{equation}\label{eq:hitting}
 |V(K)|\ge 2^d+d.
\end{equation}
\end{lemma}
\begin{proof}
Replace $K$ by the pure $d$-dimensional subcomplex generated by $A$ and $\supp z$. In its characteristic double cover $X$, write the two lifts of $A$ as $\{v_0^+,\ldots,v_d^+\}$ and $\{v_0^-,\ldots,v_d^-\}$. The closed stars $U_i^\pm=\st_X(v_i^\pm)$ cover $X$. Opposite stars have disjoint realizations, since a common vertex would contradict unique edge lifting.

Their nerve $N$ is therefore an antipodal subcomplex of the boundary of the $(d+1)$-dimensional cross-polytope. There is an equivariant simplicial map $\sd X\to N$: choose, equivariantly, a star containing each nonempty simplex of $X$. Along a chain of simplices, every chosen star contains the smallest simplex. No assertion about the homotopy type of the nerve is needed.

If a facet of the cross-polytope boundary were absent from $N$, its opposite would also be absent. After radial identification with $S^d$, choose opposite interior points $q,-q$ in these facets. Orthogonal projection to $q^\perp$, followed by normalization, defines an odd map $S^d\setminus\{q,-q\}\to S^{d-1}$. We would obtain an equivariant map $X\to S^{d-1}$. On quotients, the characteristic class would be pulled back from $\RP^{d-1}$, forcing $[a]^d=0$. For $d=1$, the quotient target is a point and the same conclusion holds.

Thus all $2^{d+1}$ intersections $\bigcap_iU_i^{\varepsilon_i}$ are nonempty. Each contains a vertex, and different intersections have disjoint vertex sets. The two constant sign intersections each contain all $d+1$ vertices of a lift of $A$. Hence
\[
 2|V(K)|\ge (2^{d+1}-2)+2(d+1)=2^{d+1}+2d.
\]
\end{proof}

The covering principle in this proof is classical; compare Bacon~\cite[property $V_n$ and Theorem~(2.23)]{Bacon}. Let $m_d$ denote the least number of vertices of a finite complex carrying a degree-one class with nonzero $d$th power. We have $m_1\ge3$ and $m_2\ge6$, the latter by Lemma~\ref{lem:hitting}. If $n<2^d+d$, a positive $d$-cycle contains two disjoint $d$-simplices. Applying Lemma~\ref{lem:deletion} twice gives
\begin{equation}\label{eq:recurrence}
 m_d\ge\min\{2^d+d,\ 2(d+1)+m_{d-2}\}\quad(d\ge3).
\end{equation}
In particular $m_3\ge11$, $m_4\ge16$, and $m_5\ge23$. We nevertheless give a separate capacity proof for twenty-two vertices in Proposition~\ref{prop:22}, since it is the first instance of the obstruction used for twenty-three vertices.

\section{Finite signed complexes}\label{sec:finite}
The following finite statements use the distinguished cocycle supplied by the signature, rather than an unspecified element of $H^1$. A signature is \emph{critical in degree four} if its fourth power is nonzero and becomes zero after deleting any vertex.

\begin{lemma}[Finite classification]\label{lem:catalogue}
Up to equivalence, the following assertions hold.
\begin{enumerate}
\item On six vertices there is one signature with $[s]^2\ne0$. Denote it by $D$; on $\{0,\ldots,5\}$ its negative edges may be taken as $13,14,24,25,35$.
\item On eleven vertices there are four signatures with $[s]^3\ne0$, each containing a balanced five-set.
\item On sixteen vertices there is one signature with $[s]^4\ne0$. Its balanced complex $E$ is pure of dimension four, with
\[
 f(E)=(16,120,330,375,150),\qquad H^1(E)=\F[s],\qquad Z_4(E)=\F\gamma_E,
\]
where $\gamma_E$ is the sum of all facets and $\ip{s^4}{\gamma_E}=1$. Each ridge has two cofacets. For every face, the facets containing it are connected by ridge adjacency through that face.
\item On seventeen vertices there are 244 signatures with $[s]^4\ne0$ and a balanced six-set. Exactly 115 are critical in degree four.
\end{enumerate}
\end{lemma}
\begin{proof}[Finite verification]
Faces and cup values are reconstructed directly from the signatures. If $\partial_r$ is the boundary matrix and $c_r$ is the vector in~\eqref{eq:cup}, the certificates for nonvanishing and vanishing are, respectively,
\begin{equation}\label{eq:cert}
 \partial_rz=0,\quad c_r^Tz=1;\qquad \partial_r^Tb=c_r.
\end{equation}
An equivalence is checked on every edge. The exhaustive reductions are as follows.

For the small catalogues, extend each signature by one vertex, switch its edge to vertex zero to zero, and enumerate every assignment on the other new edges. Explicit equivalences to the next catalogue verify coverage. The numbers of classes on three through ten vertices are
\[
 2,\ 3,\ 7,\ 16,\ 54,\ 243,\ 2038,\ 33120.
\]
On six vertices direct evaluation gives the first assertion.

For eleven vertices, first suppose a balanced five-set is present. Its complement is $D$ by deletion. Switch the five-set to zero and describe its vertices by five binary columns of length six. Repeated columns are impossible: folding one of two such vertices onto the other is a simplicial homotopy equivalence preserving the signature class, contrary to $m_3\ge11$. Sorting columns and identifying simultaneous complements leaves 3,812,256 cases. Equations~\eqref{eq:cert} leave 140 positive templates, equivalent to four signatures. In the remaining branch, the complete ten-vertex catalogue has 7,079 representatives without a balanced five-set. Every normalized one-vertex extension is checked; the 1,512,433 extensions still lacking a balanced five-set all have zero third power.

A sixteen-vertex signature with nonzero fourth power has no balanced six-set, by $m_3\ge11$. Choose a balanced five-set and apply deletion and the eleven-vertex assertion to its complement. The vertices split into blocks of sizes $6,5,5$, with the first block equal to $D$. Both eleven-vertex restrictions belong to the full list of 140 templates with the same labelled $D$. The $140^2$ pairs leave 25 unknown cross edges. Every additional balanced five-set must have a positive eleven-vertex complement. Joining these finite relations gives 1,780 assignments; 1,690 contain a forbidden balanced six-set. Of the remaining 90, 80 are equivalent to $E$ and ten have zero fourth power. The asserted invariants of $E$ follow from its reconstructed face and boundary matrices.

For seventeen vertices with a balanced six-set, deletion gives a positive eleven-vertex complement and hence a partition $D\sqcup A\sqcup B$ of sizes $6,6,5$. The twelve-vertex block $D\sqcup A$ is third-power positive. Its six columns are either distinct or have exactly one repeated pair: fewer than five distinct columns could be folded to a positive complex on at most ten vertices. Sorting and complementation leave
\[
 \frac{\binom{64}{6}+\binom{32}{3}}2+\frac{5\binom{64}{5}}2=56,550,944
\]
inputs. Exact certificates give 11,467 positive templates and 264 orbits under the block-preserving action. Combining these with all 140 labelled eleven-vertex templates gives 36,960 pairs and 30 unknown edges. Deletion at every balanced five- or six-set supplies the necessary relations. A balanced seven-set is impossible: deletion would leave a third-power-positive complex on ten vertices, contrary to $m_3\ge11$. A separate reduced ordered binary decision diagram check verifies equality of the full relation join with the saved assignment set. It contains 119,722,601 assignments; excluding balanced seven-sets leaves 85,889,170, and checking six-set complements leaves 55,551. Cup certificates and explicit equivalences give the 244 classes. Vertex deletion tests give the 115 critical classes. The normalizations never quotient the two attachments independently by automorphisms of their common labelled block.
\end{proof}

For an explicit description of $E$, list the pairs in $\binom{\{0,\ldots,15\}}2$ lexicographically, with the first pair assigned to the least significant bit. Its signature is the hexadecimal integer
\[
 \texttt{001e87420f06214c0cc944dab35bc0}.
\]

\begin{corollary}\label{cor:rigidity}
If a complex $J$ on at most sixteen vertices has a class $\alpha\in H^1(J;\F)$ with $\alpha^4\ne0$, then $J$ has sixteen vertices and is isomorphic to $E$, preserving the distinguished class.
\end{corollary}
\begin{proof}
The bound $m_4\ge16$ gives the vertex count. Any compatible completion is $E$. A positive cycle in $J$ is a nonzero element of $Z_4(E)$, hence equals $\gamma_E$. Thus $J$ contains every facet of $E$ and equals $E$.
\end{proof}

\begin{definition}
A transversal column over $W$ is a function $t:W\to\F$, identified with the set $T_t=\{(w,t(w)):w\in W\}$ in a double cover. It covers a lifted face if it contains all its vertices. Opposite columns are $t$ and $t+1$.
\end{definition}

\begin{lemma}[Sixteen-vertex covering certificates]\label{lem:Ecover}
The characteristic cover of $E$ has 300 four-dimensional facets. Each column covers at most 48 of them. Their covering number is eight; a covering which contains a pair of opposite columns has at least nine columns.
\end{lemma}
\begin{proof}[Finite verification]
Enumerating the $2^{16}$ columns gives maximum load 48. A seven-column covering would contain a column of load at least $\lceil300/7\rceil=43$. The 32 such columns form two verified automorphism orbits. Fixing a representative in each orbit leaves a six-column covering problem. A residual column with insufficient unweighted load is discarded; nonnegative rational weights on the remaining facets have total weight six and load at most one on every retained column. Equality forces an exact cover by tight columns. Complete branching on an uncovered positive-weight facet gives no solution in either case. An explicit eight-column covering verifies the upper bound.

For a prescribed opposite pair, the $2^{15}$ unordered pairs form 156 verified orbits. In 73 orbits more than $6\cdot48$ facets remain. In 77, rational weights of total greater than six and column load at most one exclude six further columns. In the remaining six orbits the total is six and complete exact-cover branching again gives no solution. In both searches, coverage of the zero-weight facets is also required at each completed leaf. Clearing denominators makes every check an integer calculation.
\end{proof}

\section{Transfer and induced cores}
Let $K$ triangulate $\RP^5$, let $a$ represent the nonzero class in $H^1(K;\F)$, and let $X\to K$ be the characteristic double cover, with involution $\tau$. Then $|X|\cong S^5$. Write $\Gamma$ for the sum of all five-dimensional facets of $X$.

We use the following local properties of a triangulated topological manifold. It is pure; each ridge has two cofacets; and the facets containing a fixed vertex are connected by adjacency through ridges containing that vertex. Indeed, for a point in the interior of a $k$-simplex $\sigma$, local homology gives
\[
 H_i(|K|,|K|\setminus\{x\})\cong\widetilde H_{i-k-1}(\lk_K\sigma).
\]
For a ridge this forces two link vertices. For a vertex, top homology of its link is one-dimensional; in a pure complex with two facets at each ridge, top cycles have constant coefficients on the components of the facet adjacency graph. Hence that graph is connected. These properties also hold in $X$. Since $X$ has no six-simplices, $Z_5(X)=\{0,\Gamma\}$.

\begin{lemma}\label{lem:transfer}
Let $z\in Z_4(K;\F)$ be positive. There exists $c\in C_5(X;\F)$ such that
\begin{equation}\label{eq:sides}
 \partial c=\Tr z,\qquad c+\tau c=\Gamma,
\end{equation}
where $\Tr$ sends a simplex to the sum of its two lifts.
\end{lemma}
\begin{proof}
Transfer is an injective chain map, and its image consists of the invariant chains. Since $H_4(X)=0$, the first equation has a solution. The chain $c+\tau c$ is a cycle, hence is zero or $\Gamma$. If it were zero, $c=\Tr b$ for some five-chain $b$ on $K$. Injectivity would imply $\partial b=z$, contradicting $\ip{a^4}{z}=1$.
\end{proof}

\begin{proposition}\label{prop:columns}
Suppose a positive four-cycle $z$ is supported in $J=K[W]$. Let $R\subseteq\supp z$ consist of faces having no five-dimensional coface in $J$. Then $|V(K)\setminus W|$ columns cover both lifts of every face in $R$. They can be chosen so that, after switching only outside $W$, their cross-edge labels and an all-zero outside block extend the actual cocycle to a compatible completion of $K$.
\end{proposition}
\begin{proof}
Choose $c$ as in~\eqref{eq:sides}. At a lift of a vertex outside $W$, no incident ridge lies in $\Tr z$. Thus adjacent incident facets have equal coefficients in $c$, and local connectivity makes this coefficient constant throughout the star. The two lifts have opposite coefficients. Choose the lift with coefficient one. Its neighbors contain at most one lift of each vertex in $W$; complete this partial choice arbitrarily to a column.

At either lift of $f\in R$, the two cofacet coefficients sum to one. The cofacet with coefficient one has an additional vertex outside $W$, since $J$ is induced. That vertex is the chosen lift of an outside vertex, and its column covers the given lift of $f$.

If $uv$ is an actual outside edge, lift it from the chosen lift of $u$ and extend it to a facet. This facet has coefficient one, so its other endpoint is the chosen lift of $v$. Switching the chosen outside lifts to zero therefore makes every actual outside edge zero. Complete the missing outside edges by zero and the missing cross edges by the chosen columns. The resulting signature preserves every actual edge.
\end{proof}

\begin{proposition}[Exclusion of twenty-two vertices]\label{prop:22}
There is no triangulation of $\RP^5$ on twenty-two vertices.
\end{proposition}
\begin{proof}
Choose a six-vertex facet $A$. Since $[a]^5\ne0$, deletion gives $[a]^4\ne0$ on $J=K[V(K)\setminus A]$, which has sixteen vertices. Corollary~\ref{cor:rigidity} gives $J=E$. Its fundamental cycle is positive, and none of its 150 facets has an internal five-dimensional coface. Proposition~\ref{prop:columns} supplies six columns covering all 300 lifted facets. Lemma~\ref{lem:Ecover} gives the contradiction
\[
 300\le6\cdot48=288.
\]
\end{proof}

For the rest of this section suppose $K$ has twenty-three vertices. No positive four-cycle can be supported on at most sixteen vertices: by Corollary~\ref{cor:rigidity} its induced support would be $E$, and Proposition~\ref{prop:columns} would cover its lifts by seven columns, contrary to Lemma~\ref{lem:Ecover}. It follows also that no compatible completion of $K$ contains a balanced seven-set. Such a set would make $a$ exact on the corresponding induced subcomplex, and deletion would give a positive four-cycle on its sixteen-vertex complement.

\begin{lemma}[Criticality of completions]\label{lem:critical}
If $|W|=17$ and $[a]^4|_{K[W]}\ne0$, every compatible completion of $K[W]$ is critical in degree four.
\end{lemma}
\begin{proof}
Suppose deleting $v\in W$ from such a completion leaves nonzero fourth power. By Lemma~\ref{lem:catalogue}, the restriction on $U=W\setminus\{v\}$ is $E$ after a single relabelling and switching. Denote this signature by $s$. Choose an actual positive cycle $z$ on $W$, and write
\[
 z=z_0+v*\partial z_0,
\]
where $z_0$ uses only vertices of $U$. Let $P$ be generated by the four-simplices in $\supp z_0$, $L$ by the support of $\partial z_0$, and $M$ by the other facets of $E$. The complex $P$ contains all sixteen vertices: otherwise $z$ would use at most sixteen. It is connected: if not, adjoining to each component its boundary coned to $v$ splits $z$ into actual cycles, one of which is positive and omits a component. The same sixteen-vertex obstruction applies. Also $M\ne\varnothing$, since otherwise $z=\gamma_E$.

We have $E=P\cup M$ and $P\cap M=L$. The latter equality holds on ridges because each ridge has two cofacets. For a lower-dimensional common face, connect a facet of $P$ to a facet of $M$ through facets containing that face; some consecutive pair meets in a mixed ridge. Thus the equality holds in every dimension. Since $P$ contains every vertex, every vertex of $M$ belongs to $L$.

Define $t(u)=a(vu)$ for \emph{every actual neighbor} $u\in U$ of $v$, and extend $t$ arbitrarily to the remaining vertices. On $L$, the triangles with vertex $v$ give $s=\delta t$. Define a cochain $s'$ on $E$ by $s'=s$ on $P$ and $s'=\delta t$ on $M$. These definitions agree on $L$, and $s'$ is a cocycle. Deleting $v$ from the positive complex $P\cup(v*L)$ gives $[s]^3|_P\ne0$, so $[s']\ne0$. Since $H^1(E)=\F[s]$, write $s'+s=\delta h$. The function $h$ is constant on the connected spanning complex $P$, hence on all vertices of $E$. Consequently
\begin{equation}\label{eq:potential}
 s|_M=\delta t|_M.
\end{equation}

Apply Lemma~\ref{lem:transfer} to the same cycle $z$. The six vertices outside $W$ give six columns on $U$. Add the opposite columns $t,t+1$. Equation~\eqref{eq:potential} shows that this pair covers both lifts of every facet of $M$. For a facet of $P$, either lift lies in $\Tr z$ and has a coefficient-one cofacet. The extra vertex cannot lie in $U$, since $E$ is four-dimensional. If it lies over $v$, the lift is covered by $t$ or $t+1$, by the choice of $t$ on all actual neighbors. Otherwise it is covered by one of the six outside columns. We have covered all 300 facets with eight columns including an opposite pair, contradicting Lemma~\ref{lem:Ecover}.
\end{proof}

\section{The seventeen-vertex obstruction}
We first give the certificate criterion which ensures that faces of a completion are present in an actual positive cycle.

\begin{lemma}\label{lem:forced}
Let $f$ be a four-dimensional face of $B(s)$. If a three-cochain $b_f$ satisfies
\begin{equation}\label{eq:forced}
 \delta b_f=s^4+\mathbf1_f,
\end{equation}
then every positive four-cycle contains $f$, and $f$ has no five-dimensional coface.
\end{lemma}
\begin{proof}
Pairing~\eqref{eq:forced} with a positive cycle gives its coefficient at $f$ equal to one. Taking a coboundary gives $\delta\mathbf1_f=0$; a five-simplex containing $f$ would give value one, a contradiction.
\end{proof}

The same conclusion follows if $H_4(B(s))=\F$, $z_*$ is a positive cycle, and $f\in\supp z_*$ has no five-dimensional coface. Indeed, any other positive cycle differs from $z_*$ by a boundary, whose coefficient at $f$ is zero. This is the criterion used by the independent rank verification.

For a set $R$ of such faces and nonnegative integer weights $w_f$, put
\begin{equation}\label{eq:capacity}
 W=\sum_{f\in R}w_f,\qquad
 \ell(t)=\sum_{\substack{f\in R\\T_t\text{ contains a lift of }f}}w_f,
 \qquad L=\max_t\ell(t).
\end{equation}
A column contains at most one of the two lifts of a nonempty face. Six columns covering both lifts of every face in $R$ must therefore satisfy $2W\le6L$.

\begin{lemma}[Critical covering certificates]\label{lem:115}
Each of the 115 critical signatures in Lemma~\ref{lem:catalogue} admits a family $R$ of forced faces without five-dimensional cofaces and integer weights as above. For 104 signatures, $2W>6L$. For each of the other eleven, there is no six-column covering of the two lifts of $R$ whose extension by an all-zero outside block contains no balanced seven-set.
\end{lemma}
\begin{proof}[Finite verification]
For every signature, reconstruct the boundary matrices, verify $\dim H_4=1$, and check a positive cycle. Its faces without five-dimensional cofaces give the required forced faces. Evaluate~\eqref{eq:capacity} on all $2^{17}$ columns. The strict inequality holds for 104 signatures. For example, two certificates have $(W,L)=(611,201)$ and $(62,20)$, giving $1222>1206$ and $124>120$.

The eleven remaining indices in the critical catalogue are
\[
 3,16,18,39,42,53,54,59,67,86,105.
\]
For these certificates, $5L<2W\le6L$. Set $\Delta=6L-2W$. A balanced seven-set consisting of $k$ outside vertices and $7-k$ core vertices exists precisely when the corresponding $k$ columns contain the same lift of a balanced $(7-k)$-set. This condition is checked for every $1\le k\le6$. In particular, two columns cannot cover the same lift of any positive-weight face. A valid covering would thus satisfy
\begin{equation}\label{eq:deficit}
 \sum_{j=1}^6\ell(t_j)=2W,\qquad
 \sum_{j=1}^6(L-\ell(t_j))=\Delta.
\end{equation}
Every selected column has load at least $L-\Delta$.

Enumerate all columns, retain those satisfying this load bound and the one-column restriction, and branch on an uncovered lift of a face in $R$, including faces of weight zero. Every possible completion contains a candidate covering this lift, so branching over all such candidates is exhaustive. After adjoining a column, check the remaining deficit and every new common-lift restriction. Reject a branch if a required lift has no candidate, the deficit budget is exceeded, or six selected columns leave a required lift uncovered. All retained loads are positive and $2W>5L$, so no valid covering repeats a column or ends with fewer than six. This computation has 1,012 nodes and no successful leaf; 997 nodes belong to index 67. Every rejection is checked from its explicit missing lift, deficit, or balanced seven-set. No symmetry restriction on the order of selected columns is used.
\end{proof}

\begin{proof}[Proof of Theorem~\ref{thm:main}]
Equation~\eqref{eq:AM} excludes at most twenty-one vertices, and Proposition~\ref{prop:22} excludes twenty-two. Suppose that $K$ has twenty-three vertices. Since $[a]^5\ne0$ and $23<2^5+5$, Lemma~\ref{lem:hitting} gives two disjoint six-vertex simplices $A,C$ in a positive five-cycle. Deleting $A$ gives an actual induced seventeen-vertex complex
\[
 J=K[V(K)\setminus A],\qquad [a]^4|_J\ne0,
\]
which contains $C$. Any compatible completion of $J$ is fourth-power positive, contains a balanced six-set, and is critical by Lemma~\ref{lem:critical}. Thus it belongs to one of the 115 classes.

Transport its signature, lifted faces, and columns by the same permutation and switching. The forced faces of Lemma~\ref{lem:115} occur in every actual positive cycle and have no internal five-dimensional coface. Proposition~\ref{prop:columns} supplies six columns covering both lifts of each. In 104 classes this contradicts $2W>6L$. In the other eleven, complete the outside block by zero as in Proposition~\ref{prop:columns}. No compatible completion contains a balanced seven-set, by the sixteen-vertex obstruction following Proposition~\ref{prop:22}. The last assertion of Lemma~\ref{lem:115} gives the contradiction. Hence at least twenty-four vertices are necessary. The construction in~\cite{GSY} attains this number.
\end{proof}

\section*{Computational verification}
The computational part consists of the classification in Lemma~\ref{lem:catalogue} and the covering certificates in Lemmas~\ref{lem:Ecover} and~\ref{lem:115}. Cup products, boundary equations, equivalence maps, and loads are checked with exact arithmetic. Catalogue coverage is checked separately from the properties of the listed representatives. In particular, the seventeen-vertex assertion concerns signatures containing a balanced six-set; no classification of the complementary branch is used.

The compact ancillary files verify the final signed complexes, covering bounds, and residual search certificates. The exhaustive classification records are retained separately; the accompanying manifest identifies their inputs, programs, and verification records. The present manuscript uses those recorded classification runs and does not assert a fresh execution of every enumeration. Several intermediate statements have Lean~4 proofs, but the global theorem is not presently an end-to-end Lean formalization.

\section*{Acknowledgements}
The author thanks his advisor, Jiming Ma, and Fangting Zheng for their guidance.

\section*{Computational assistance}
OpenAI Codex assisted with proof exploration, literature searches, verification code, and preparation of the manuscript. The mathematical argument and the finite verification procedures are specified above and in the accompanying files.

\end{document}